\documentclass[11pt]{amsart}

\usepackage[T1]{fontenc}
\usepackage[utf8]{inputenc}
\usepackage{lmodern}
\usepackage{amsmath,amssymb,amsfonts,amsthm,mathtools}
\usepackage{enumerate}
\usepackage{microtype}
\usepackage{cite}
\usepackage[hidelinks]{hyperref}

\newtheorem{theorem}{Theorem}[section]
\newtheorem{proposition}[theorem]{Proposition}
\newtheorem{lemma}[theorem]{Lemma}
\newtheorem{corollary}[theorem]{Corollary}

\theoremstyle{definition}

\theoremstyle{remark}
\newtheorem{remark}[theorem]{Remark}

\newcommand{\F}{\mathbb F}
\newcommand{\K}{\mathbb K}
\newcommand{\PP}{\mathbb P}
\newcommand{\W}{\mathcal W}
\newcommand{\X}{\mathcal X}

\DeclareMathOperator{\Aut}{Aut}
\DeclareMathOperator{\Tr}{Tr}
\DeclareMathOperator{\gon}{gon}

\title[Maximal curves with trivial automorphism group]
      {Maximal curves with trivial automorphism group}

\author[D. Bartoli]{Daniele Bartoli}
\address{Dipartimento di Matematica e Informatica,
Universit\`a degli Studi di Perugia, Perugia, Italy}
\email{daniele.bartoli@unipg.it}

\author[H. Borges]{Herivelto Borges}
\address{Instituto de Matem\'atica e Estat\'istica,
Universidade de S\~ao Paulo, S\~ao Paulo, Brazil}
\email{hborges@ime.usp.br}

\author[M. Timpanella]{Marco Timpanella}
\address{Dipartimento di Matematica e Informatica,
Universit\`a degli Studi di Perugia, Perugia, Italy}
\email{marco.timpanella@unipg.it}

\author[G. Zini]{Giovanni Zini}
\address{Dipartimento di Scienze Fisiche, Informatiche e Matematiche,
Universit\`a degli Studi di Modena e Reggio Emilia, Modena, Italy}
\email{giovanni.zini@unimore.it}

\date{}


\begin{document}

\begin{abstract}
We construct algebraic curves with trivial geometric automorphism group by
taking cyclic scalar quotients of curves of the form
\[
 y^r-y=xR(x),\qquad r=p^e.
\]
We give a condition on the genus
under which the quotient has no nontrivial automorphism.  The proof uses the
ramification of a degree-$r$ morphism to the projective line,
Castelnuovo--Severi, and explicit automorphisms of the covering curve.  The
quotient also has gonality $r$, with a unique morphism of degree $r$ to the
projective line up to an automorphism of the latter.  Applying the construction
to maximal additive curves gives explicit maximal curves with trivial
geometric automorphism group in every positive characteristic, including
families defined over the prime field and of unbounded genus. To the best
of our knowledge, these are the first examples of maximal curves with
trivial geometric automorphism group. Every maximal curve in these families
is covered by the Hermitian curve. For every odd prime $p$ and every
$n\ge2$, we prove that the curve $\X_{p,n}$, maximal over
$\F_{p^{4n+2}}$, is not a Galois subcover of the corresponding
Hermitian curve, even over the algebraic closure. Thus every odd
characteristic yields an infinite family of curves covered but not
Galois covered by the Hermitian curve. To the best of our knowledge,
these are the first such infinite families.
\end{abstract}

\maketitle

\section{Introduction}

Let $\K=\overline{\F}_p$.  For a power $q$ of $p$, a curve $\mathcal C$
defined over $\F_{q^2}$ is called $\F_{q^2}$-maximal if
\[
 \#\mathcal C(\F_{q^2})=q^2+1+2qg(\mathcal C).
\]
Unless otherwise stated, curves are nonsingular, projective and absolutely
irreducible.  We consider their automorphism groups over $\K$.

The Hermitian curve over $\F_{q^2}$ is the basic example of a maximal curve:
it has genus $q(q-1)/2$, the largest possible genus of an
$\F_{q^2}$-maximal curve.  Its quotients have provided many further examples
\cite{HKT}.  More generally, a quotient defined over $\F_{q^2}$ of an
$\F_{q^2}$-maximal curve is again maximal \cite[Proposition~6]{Lachaud}.
This method also applies to maximal curves outside the class of Hermitian
subcovers; the curve of Giulietti and Korchm\'aros is a well-known example
which is not covered by the Hermitian curve over its field of definition
\cite{GK}.

Curves with trivial geometric automorphism group also exist over very small
finite fields: Poonen proved existence over every field in every genus
$g\ge3$ \cite{Poonen}.  Moreover, Achter, Glass and Pries showed that over
an algebraically closed field of characteristic $p$, every genus $g\ge3$
and every $p$-rank between $0$ and $g$ occur for such a curve \cite{AGP}.
Neither result prescribes the number of rational points over a finite field.
The problem considered here is to give equations for curves which are
maximal over a specified finite field and have no automorphisms even over
$\K$.  To the best of our knowledge, the curves constructed below are the
first examples in the literature combining maximality with a trivial
geometric automorphism group.  Passing from a maximal curve to a quotient
ensures maximality, but
determining the full automorphism group of the quotient requires a separate
argument: its automorphisms need not lift to the covering curve.  We obtain
that argument from a distinguished rational subfield of small index.

Our starting point is the curve
\begin{equation}\label{eq:intro-cover}
 \W_R:\qquad y^r-y=xR(x),\qquad r=p^e,
\end{equation}
where $R$ is an $r$-linearized polynomial.  Curves of this form were studied
by van der Geer and van der Vlugt in characteristic two \cite{vdGvdV}, and by
Bouw, Ho, Malmskog, Scheidler, Srinivasan and Vincent in odd characteristic
\cite{BouwEtAl}.  The prime-power additive degree is also considered in
\cite[Remark~13.8]{vdGvdV}.  Besides translations in $y$, these curves admit
scalar automorphisms for suitable $R$.

We first study quotient curves without imposing any condition on the number of
rational points.  We then choose maximal curves of the form
\eqref{eq:intro-cover}; maximality passes to their quotients defined over the
same field.

More precisely, write
\begin{equation}\label{eq:intro-R}
 R(X)=\sum_{j=0}^{h}B_jX^{r^j},\qquad B_0B_h\ne0,
\end{equation}
and define
\begin{equation}\label{eq:intro-D}
 D=
 \begin{cases}
 2(r-1),& r\text{ odd and }B_j=0\text{ for every odd }j,\\
 r-1,&\text{otherwise}.
 \end{cases}
\end{equation}
The group
\[
 C_D=\{(x,y)\mapsto(ax,a^2y):a^D=1\}
\]
acts on $\W_R$.  The curve of interest is the quotient
\[
 \X_R=\W_R/C_D.
\]
Put $\epsilon=\gcd(2,D)$.  The quotient morphism is described explicitly by
$u=x^D$ and $z=y/x^2$.  In particular, $u$ defines a morphism of degree
$r$ from $\X_R$ to the projective line, and
\begin{equation}\label{eq:intro-genus}
 g(\X_R)=\frac{(r-1)(r^h+1-\epsilon)}{2D}.
\end{equation}
Our main geometric result is the following.

\begin{theorem}\label{thm:intro-rigidity}
Assume that $r=p^e>2$, $D>\epsilon$, and
\[
 g(\X_R)>(r-1)^2.
\]
Then
\[
 \Aut_{\K}(\X_R)=1.
\]
Moreover,
\[
 \gon(\X_R)=r,
\]
and $u:\X_R\to\PP^1$ is the unique morphism of degree $r$ to the projective
line, up to an automorphism of $\PP^1$.
\end{theorem}

The arithmetic part uses the point count for curves of the form
\eqref{eq:intro-cover}, expressed through a quadratic form
\cite{vdGvdV,BouwEtAl,CakcakOzbudak}.  Let $r=p^e>3$, let $m\ge4$, put $q=r^m$, and
choose $\alpha,t\in\F_{q^2}$ with
\[
 \alpha\ne0,\qquad t^q=-t,\qquad t^r+t\ne0.
\]
Set
\begin{equation}\label{eq:intro-Rat}
 R_{\alpha,t}(X)=\alpha^2tX+
 \sum_{j=1}^{m-1}\alpha^{r^j+1}(t^{r^j}+t)X^{r^j}.
\end{equation}
For an $r$-linearized polynomial $R$, define $\Psi_R$ by
\begin{equation}\label{eq:intro-Psi}
 R(X)=X\Psi_R(X^{r-1}).
\end{equation}

\begin{theorem}\label{thm:intro-maximal}
With the notation above, the nonsingular projective model of
\begin{equation}\label{eq:intro-maximal-curve}
 \X_{\alpha,t}:\qquad u^2z^r-z=\Psi_{R_{\alpha,t}}(u)
\end{equation}
is $\F_{q^2}$-maximal and satisfies
\[
 \Aut_{\K}(\X_{\alpha,t})=1.
\]
Its genus is
\[
 g(\X_{\alpha,t})=
 \begin{cases}
 (r^{m-1}-1)/2,&p\ne2,\\[1mm]
 r^{m-1}/2,&p=2.
 \end{cases}
\]
\end{theorem}

The case $r=p\ge5$ gives genus $(q-p)/(2p)$.  In characteristic two we
take $r=4$.  Characteristic three is obtained by using the larger scalar
group arising from an even-support polynomial.  In fact, the construction
produces equations over the prime field in every characteristic.

\begin{theorem}\label{thm:intro-prime-field}
Let $n\ge2$.
\begin{enumerate}[\upshape (i)]
\item The curve over $\F_2$
\begin{equation}\label{eq:intro-binary}
 \X_{2,n}:\qquad
 u^2v^4+v
 =1+\sum_{\substack{1\le j\le3n-1\\3\nmid j}}
 u^{(4^j-1)/3}
\end{equation}
is maximal over $\F_{2^{12n}}$, has genus $2^{6n-3}$, and
\[
 \Aut_{\overline{\F}_2}(\X_{2,n})=1.
\]
\item For every odd prime $p$, the curve over $\F_p$
\begin{equation}\label{eq:intro-odd}
 \X_{p,n}:\qquad
 uv^p+v=-1-2\sum_{i=1}^{n}
 u^{(p^{2i}-1)/(2(p-1))}
\end{equation}
is maximal over $\F_{p^{4n+2}}$, has genus
\[
 \frac{p^{2n}-1}{4},
\]
and
\[
 \Aut_{\overline{\F}_p}(\X_{p,n})=1.
\]
\end{enumerate}
\end{theorem}

Since the automorphism group of the Hermitian curve $\mathcal{H}_q$ is large and has a rich subgroup lattice, many $\mathbb{F}_{q^2}$-maximal curves have been constructed as Galois subcovers $\mathcal{H}_q/G$. On the contrary, the natural and long-standing question on the existence of (non-rational, non-elliptic) $\F_{q^2}$-maximal subcovers of $\mathcal H_q$ which are not birationally equivalent to Galois subcovers of $\mathcal{H}_q$ was not answered until a very recent preprint \cite{MaWang}.
Curves with trivial geometric automorphism
group are natural candidates. Indeed, if $\mathcal X$ is isomorphic to
$\mathcal H_q/G$, then
\[
 N_{\Aut_{\K}(\mathcal H_q)}(G)/G
 \hookrightarrow \Aut_{\K}(\mathcal X);
\]
therefore, if $\Aut_{\K}(\mathcal X)=1$, the group $G$ must be
self-normalizing.
The preprint \cite{MaWang} by Ma and Wang provides an $\mathbb{F}_{27^2}$-maximal curve of genus $2$ which is 
the first known sporadic example of a maximal curve covered but not Galois covered by the corresponding
Hermitian curve.
Such curve is $\X_{3,1}$, obtained by allowing $n=1$ in
\eqref{eq:intro-odd}. Indeed, the equation $uv^3+v=u^2-1$ becomes their
equation $w^2=t^6-t+1$ upon setting $t=-v$ and $w=2u-v^3$.
This first member is outside the range of
Theorem~\ref{thm:intro-prime-field}; our result below extends the
sporadic example to the first known infinite families, one in each odd
characteristic.

All the maximal curves constructed in this paper are covered by the
Hermitian curve over the corresponding field of maximality. The nature
of this covering depends on the scalar subgroup. In
Section~\ref{sec:hermitian-coverings} we prove that it is Galois for
$D=r-1$. For $D=2(r-1)$ we obtain the following result.

\begin{theorem}\label{thm:uniform-nongalois}
Let $p$ be an odd prime, $n\ge2$, $m=2n+1$, and $q=p^m$.
The curve $\X_{p,n}$ of \eqref{eq:intro-odd} is covered by
$\mathcal H_q$ over $\F_{q^2}$, but is not isomorphic over $\K$ to
$\mathcal H_q/G$ for any subgroup $G\le\Aut_{\K}(\mathcal H_q)$.
\end{theorem}

The proof uses the genus and the triviality of the geometric
automorphism group of $\X_{p,n}$ to restrict possible Galois groups. As $n$
varies, the genera $(p^{2n}-1)/4$ strictly increase, giving infinitely
many pairwise geometrically nonisomorphic curves in each odd
characteristic. Here the conclusion is geometric:
there is no Galois morphism even over $\K$.

Section~\ref{sec:scalar} gives the quotient equation and its ramification.
In Section~\ref{sec:rigidity} we prove that the quotient has trivial
automorphism group under the stated genus condition.  Section~\ref{sec:maximal-covers}
contains the point-counting input, and Sections~\ref{sec:maximal-quotients}
and~\ref{sec:prime-field} give the maximal families and the equations over
prime fields.  Section~\ref{sec:hermitian-coverings} treats their
Hermitian coverings.

\section{Scalar quotients of additive curves}
\label{sec:scalar}

Throughout this section we work over $\K$; no maximality assumption is made.
Let $r=p^e>2$ and let
\begin{equation}\label{eq:R-general}
 R(X)=\sum_{j=0}^{h}B_jX^{r^j}\in \K[X],
 \qquad h\ge1,\qquad B_0B_h\ne0.
\end{equation}
Let $\W_R$ be the nonsingular projective model of the affine curve
\begin{equation}\label{eq:W-general}
 y^r-y=xR(x).
\end{equation}
Since $\deg(xR(x))=r^h+1$ is prime to $p$, the standard theory of one-pole
additive extensions gives
\begin{equation}\label{eq:genus-W}
 g(\W_R)=\frac{(r-1)r^h}{2};
\end{equation}
see, for example, \cite[Proposition~3.7.10]{Stichtenoth}.

Define $D$ as in \eqref{eq:intro-D}.  Then
\begin{equation}\label{eq:CD}
 C_D=\{\rho_a:(x,y)\mapsto(ax,a^2y):a^D=1\}
\end{equation}
acts on $\W_R$.  Indeed, $a^2\in\F_r$ and
$a^{r^j-1}=1$ whenever $B_j\ne0$, so that
\[
 (ax)R(ax)=a^2xR(x).
\]
When $r=p$ is odd, the corresponding scalar subgroup occurs in the description of the stabilizer of the point at infinity in
\cite[Theorem~13.3]{vdGvdV} and \cite[Theorem~1.4.3]{BouwEtAl}.

Let
\[
 \X_R=\W_R/C_D,
 \qquad \epsilon=\gcd(2,D).
\]
Define $\Psi_{R,D}$ by
\begin{equation}\label{eq:Psi-D}
 R(X)=X\Psi_{R,D}(X^D),
\end{equation}
that is,
\[
 \Psi_{R,D}(U)=B_0+
 \sum_{\substack{1\le j\le h\\B_j\ne0}}
 B_jU^{(r^j-1)/D}.
\]
The exponents are integers by the definition of $D$.

\begin{proposition}\label{prop:quotient}
The quotient curve $\X_R=\W_R/C_D$ has function field generated by
\[
 u=x^D,\qquad z=\frac{y}{x^2}
\]
and these functions satisfy
\begin{equation}\label{eq:quotient-general}
 u^{2(r-1)/D}z^r-z=\Psi_{R,D}(u).
\end{equation}
Moreover,
\begin{equation}\label{eq:genus-X}
 g(\X_R)=\frac{(r-1)(r^h+1-\epsilon)}{2D}.
\end{equation}
The function $u$ defines a morphism
\[
 u:\X_R\longrightarrow\PP^1
\]
of degree $r$.  The fibre over infinity consists of one point of ramification
index $r$.  The fibre over zero contains one unramified point and
$(r-1)\epsilon/D$ points of ramification index $D/\epsilon$.  No other value is
a branch point.
\end{proposition}

\begin{proof}
The functions $u$ and $z$ are fixed by $C_D$.  Since $x^D=u$ and
$y=x^2z$, one has
\[
 \K(\W_R)=\K(u,z)(x),
\]
and the $D$ scalar automorphisms show that $\K(u,z)$ is exactly the fixed
field of $C_D$.  Substituting $y=x^2z$ into \eqref{eq:W-general} gives
\eqref{eq:quotient-general}.

For the tame quotient morphism $\W_R\to\X_R$, the point at infinity and $(0,0)$ have
stabilizer $C_D$.  Each point $(0,b)$, with $b\in\F_r^\ast$, has stabilizer
of order $\epsilon$, and all other stabilizers are trivial.  Riemann--Hurwitz,
together with \eqref{eq:genus-W}, gives \eqref{eq:genus-X}.  The stated
ramification of $u$ follows from the tower
\[
 \W_R\longrightarrow\X_R\longrightarrow\PP^1_u,
 \qquad u=x^D.
\]
\end{proof}

\section{Rigidity of the quotient}
\label{sec:rigidity}

We determine the full automorphism group of $\X_R$.  If the genus is
sufficiently large, the morphism $u:\X_R\to\PP^1$ is preserved by every
automorphism of $\X_R$.  Its two branch fibres then restrict the possible
induced automorphisms of $\PP^1$.
Put
\[
 L=\K(\W_R),\qquad F=\K(\X_R).
\]
The composite morphism $\W_R\to\X_R\to\PP^1$ is Galois; equivalently,
$L/\K(u)$ is Galois with group
\begin{equation}\label{eq:G}
 G=V\rtimes C_D,
 \qquad
 V=\{(x,y)\mapsto(x,y+b):b\in\F_r\}\simeq(\F_r,+).
\end{equation}
Conjugation by $C_D$ acts on $V$ through multiplication by
\[
 H=\{a^2:a^D=1\}\subseteq\F_r^\ast.
\]

\begin{lemma}\label{lem:no-intermediate}
The morphism $u:\X_R\to\PP^1$ has no nontrivial factorization through an
intermediate curve.  Equivalently, there is no proper intermediate field
between $\K(u)$ and $F$.
\end{lemma}

\begin{proof}
The subgroup $H$ generates $\F_r$ over $\F_p$.  Indeed, if $r$ is even or
$D=2(r-1)$, then $H=\F_r^\ast$.  In the remaining case $H$ is the subgroup
of squares.  If $r$ is not prime, its order $(r-1)/2$ is larger than
$\sqrt r-1$, whereas the multiplicative group of a proper subfield has at
most $\sqrt r-1$ elements.  The prime case is immediate.

If $C_D\le J\le G$, then $J=(J\cap V)C_D$, and $J\cap V$ is stable under
$H$.  Hence $J\cap V$ is an $\F_r$-subspace of $V=\F_r$, so it is either
zero or all of $V$.  Galois correspondence gives the result.
\end{proof}

Let $f_1,f_2:\mathcal C\to\PP^1$ be morphisms of degrees $d_1,d_2$.
If $\K(\mathcal C)=\K(f_1,f_2)$, the Castelnuovo--Severi inequality gives
\begin{equation}\label{eq:CS}
 g(\mathcal C)\le(d_1-1)(d_2-1);
\end{equation}
see \cite[Theorem~3.11.3]{Stichtenoth}.

\begin{lemma}\label{lem:pencil}
Assume
\begin{equation}\label{eq:rigidity-hyp}
 D>\epsilon,
 \qquad
 g(\X_R)>(r-1)^2.
\end{equation}
Then every $\sigma\in\Aut_{\K}(\X_R)$ satisfies
\[
 \sigma(u)=\lambda u
\]
for some $\lambda\in \K^\ast$.  Moreover, $\sigma$ lifts to an automorphism
$\widetilde\sigma$ of $\W_R$ such that
\[
 \widetilde\sigma(x)=ax,
 \qquad a^D=\lambda,
\]
and the lift centralizes $C_D$.
\end{lemma}

\begin{proof}
The rational subfields $\K(u)$ and $\K(\sigma(u))$ both have index $r$ in $F$.
If they were distinct, Lemma~\ref{lem:no-intermediate} would give
\[
 F=\K(u,\sigma(u)),
\]
and \eqref{eq:CS} would imply $g(\X_R)\le(r-1)^2$, contrary to
\eqref{eq:rigidity-hyp}.  Thus $\sigma$ preserves $\K(u)$ and induces an
automorphism of the quotient curve $\PP^1$.

By Proposition~\ref{prop:quotient}, the only branch points of $u$ are zero
and infinity, and the ramification types of their fibres are different because
$D>\epsilon$.  Hence the induced M\"obius transformation fixes both values,
so $\sigma(u)=\lambda u$.

Since $L=F(x)$ and $x^D=u$, choosing $a^D=\lambda$ gives a lift with
$x\mapsto ax$.  This lift normalizes $C_D$.  Conjugation preserves the scalar
by which an element of $C_D$ acts on $x$, and therefore fixes each element of
$C_D$.
\end{proof}

\begin{theorem}\label{thm:rigidity}
Under the hypotheses \eqref{eq:R-general}, \eqref{eq:intro-D}, and
\eqref{eq:rigidity-hyp}, one has
\[
 \Aut_{\K}(\X_R)=1.
\]
\end{theorem}

\begin{proof}
Let $\sigma\in\Aut_{\K}(\X_R)$ and let $\widetilde\sigma$ be the lift from
Lemma~\ref{lem:pencil}, with $\widetilde\sigma(x)=ax$.

When $r=p$ is odd, the description of automorphisms fixing the point at
infinity in \cite[Proposition~1.3.2 and Lemma~1.4.1]{BouwEtAl} gives
\begin{equation}\label{eq:lift-form}
 \widetilde\sigma(y)=dy+b,
 \qquad d\in\F_r^\ast,\quad b\in\F_r,
 \qquad aR(aX)=dR(X).
\end{equation}
For $r=p^e$ with $e>1$, the same form follows as follows.  The lift normalizes
the translation group $V$ and centralizes $C_D$.  Its induced action on $V$
is therefore an $\F_p$-linear map commuting with $H$, hence is
$\F_r$-linear.  Thus
\[
 \widetilde\sigma(y)=dy+B(x),
 \qquad d\in\F_r^\ast,
 \]
with $B(x)\in \K(x)$.  The defining equation gives
\begin{equation}\label{eq:B-comparison}
 (ax)R(ax)-dxR(x)=B(x)^r-B(x).
\end{equation}
A finite pole of $B$ is impossible, so $B\in \K[x]$.  If $B$ were
nonconstant, the degree on the right would be divisible by $r$, whereas all
possible exponents on the left are $2$ and $r^j+1$.  Hence $B=b$ is constant,
and \eqref{eq:lift-form} follows.

Since $\widetilde\sigma$ centralizes $C_D$, one has
\[
 b=a_0^2b
\]
for every $a_0$ with $a_0^D=1$.  The condition $D>\epsilon$ gives an element
with $a_0^2\ne1$, so $b=0$.  Comparing the coefficient of $X$ in
$aR(aX)=dR(X)$ gives $d=a^2$.

If $r$ is even, $a^2\in\F_r$ implies $a\in\F_r^\ast$, hence $a^D=1$.
If $r$ is odd and $D=2(r-1)$, the same conclusion follows from
$a^2\in\F_r^\ast$.  Finally, suppose that $r$ is odd and $D=r-1$.  By the
definition of $D$ there is an odd index $j$ with $B_j\ne0$.  Comparing that
coefficient gives $a^{r^j-1}=1$.  Since
\[
 \frac{r^j-1}{r-1}=1+r+\cdots+r^{j-1}
\]
is odd, while $a^{2(r-1)}=1$, we obtain $a^{r-1}=1$.  Thus in every case
$\widetilde\sigma\in C_D$, so its action on $F=L^{C_D}$ is trivial.
\end{proof}

\begin{corollary}\label{cor:gonality}
Under the hypotheses of Theorem~\ref{thm:rigidity},
\[
 \gon(\X_R)=r.
\]
Moreover, the map $u:\X_R\to\PP^1$ is the unique degree-$r$ map to the
projective line, up to an automorphism of $\PP^1$.
\end{corollary}

\begin{proof}
The morphism $u$ has degree $r$.  Suppose that $v:\X_R\to\PP^1$ has degree
$d<r$.  If $\K(u,v)\ne F$, Lemma~\ref{lem:no-intermediate} gives
$\K(u,v)=\K(u)$.  But then $v\in\K(u)$ and $d$ is a multiple of $r$,
a contradiction.  Hence $F=\K(u,v)$, and Castelnuovo--Severi yields
\[
 g(\X_R)\le(r-1)(d-1)< (r-1)^2,
\]
a contradiction.  Hence $\gon(\X_R)=r$.

If $v$ has degree $r$ and $\K(v)\ne\K(u)$, then $\K(u,v)=F$ by the same
argument and Castelnuovo--Severi gives $g(\X_R)\le(r-1)^2$.  Therefore
$\K(v)=\K(u)$, which is equivalent to the stated
uniqueness.
\end{proof}

\begin{remark}
Since $\Aut_{\K}(\X_R)=1$, the curve admits no nontrivial separable Galois
covering onto another curve.  Thus the construction gives curves of
unbounded genus with a distinguished low-degree pencil but without geometric
symmetries.
\end{remark}

\section{Maximal additive covers}
\label{sec:maximal-covers}

We now impose arithmetic conditions on the covering curve.  The preceding
results on the automorphisms of $\X_R$ did not require maximality.

The number of rational points of a curve
\[
 y^r-y=xR(x)
\]
over $\F_{q^2}$ can be expressed through
\[
 x\longmapsto\Tr_{\F_{q^2}/\F_r}(xR(x)).
\]
We use the following known criterion for prime-power $r$
\cite[Theorem~3.1]{CakcakOzbudak}; for $r=p$, compare
\cite[Proposition~13.4]{vdGvdV} and
\cite[Proposition~1.2.6]{BouwEtAl}.

\begin{proposition}\label{prop:maximal-rank-two}
Let $m\ge2$, $q=r^m$, and let $R\in\F_{q^2}[X]$ be $r$-linearized of
degree $r^{m-1}$.  If
\[
 x\longmapsto\Tr_{\F_{q^2}/\F_r}(xR(x))
\]
is a product of two independent $\F_r$-linear forms, then the curve
\[
 y^r-y=xR(x)
\]
is $\F_{q^2}$-maximal.
\end{proposition}

We apply this criterion to the following family.  Put
\[
 q=r^m,\qquad T=\Tr_{\F_{q^2}/\F_r}.
\]

\begin{lemma}\label{lem:trace-family}
Let $\alpha\in\F_{q^2}^\ast$ and $t\in\F_{q^2}$ satisfy $t^q=-t$.
Define
\begin{equation}\label{eq:R-alpha-t}
 R_{\alpha,t}(X)=\alpha^2tX+
 \sum_{j=1}^{m-1}
 \alpha^{r^j+1}(t^{r^j}+t)X^{r^j}.
\end{equation}
Then
\begin{equation}\label{eq:trace-product}
 T(xR_{\alpha,t}(x))
 =T(\alpha x)T(\alpha t x)
\end{equation}
for every $x\in\F_{q^2}$.
\end{lemma}

\begin{proof}
Put $A=\alpha x$.  Expanding the product of traces and using the invariance of
$T$ under Frobenius gives
\[
 T(A)T(tA)=\sum_{j=0}^{2m-1}T(t^{r^j}A^{r^j+1}).
\]
For $1\le j<m$, the terms indexed by $j$ and $2m-j$ combine to
\[
 T((t^{r^j}+t)A^{r^j+1}).
\]
The term with index zero is $T(tA^2)$, and the middle term vanishes because
$A^{q+1}\in\F_q$ and $t^q+t=0$.  This is exactly
\eqref{eq:trace-product}.
\end{proof}

\begin{corollary}\label{cor:maximal-parent}
Assume, in addition, that $t\ne0$ and $t^r+t\ne0$.  Then
$R_{\alpha,t}$ has degree $r^{m-1}$, and
\[
 \W_{\alpha,t}:\qquad y^r-y=xR_{\alpha,t}(x)
\]
is $\F_{q^2}$-maximal.
\end{corollary}

\begin{proof}
The two trace forms in \eqref{eq:trace-product} are independent: otherwise
$t\in\F_r$.  Moreover, if the coefficient of $X^{r^{m-1}}$ vanished, then
$t^{r^{m-1}}=-t$; taking $r$th powers and using $t^q=-t$ would give
$t^r=t$, contradicting the hypotheses.  Proposition~\ref{prop:maximal-rank-two}
now applies.
\end{proof}

We also use the fact that an $\F_{q^2}$-quotient of an
$\F_{q^2}$-maximal curve is $\F_{q^2}$-maximal
\cite[Proposition~6]{Lachaud}.

\section{Maximal quotients with trivial automorphism group}
\label{sec:maximal-quotients}

We now apply the geometric construction to the maximal covering curves of the
preceding section.

\begin{proof}[Proof of Theorem~\ref{thm:intro-maximal}]
The coefficients of $R_{\alpha,t}$ satisfy
\[
 B_0=\alpha^2t\ne0,
 \qquad
 B_1=\alpha^{r+1}(t^r+t)\ne0,
\]
and the leading coefficient is nonzero by Corollary~\ref{cor:maximal-parent}.
Hence $D=r-1$.  Proposition~\ref{prop:quotient} gives
\eqref{eq:intro-maximal-curve} and the stated genus.

Since $r>3$ and $m\ge4$,
\[
 D>\gcd(2,D),
 \qquad
 g(\X_{\alpha,t})>(r-1)^2.
\]
Theorem~\ref{thm:rigidity} therefore gives
\[
 \Aut_{\K}(\X_{\alpha,t})=1.
\]
By Corollary~\ref{cor:maximal-parent}, the covering curve is
$\F_{q^2}$-maximal.  The subgroup $C_D$ is defined over $\F_{q^2}$, so its
quotient $\X_{\alpha,t}$ is also $\F_{q^2}$-maximal.
\end{proof}

\begin{remark}\label{rem:pge5}
For $r=p\ge5$ and $q=p^m$, Theorem~\ref{thm:intro-maximal} gives
\[
 g(\X_{\alpha,t})=\frac{q-p}{2p}.
\]
Thus every fixed characteristic $p\ge5$ contains families of
$\F_{q^2}$-maximal curves of unbounded genus with trivial geometric
automorphism group.
\end{remark}

The family of Theorem~\ref{thm:intro-maximal} does not include $p=3$ with
$r=p$.  For the prime-field family below, a polynomial involving only the
 powers $X^{p^j}$ with even $j$ permits the larger scalar group of order $2(p-1)$.
The resulting quotient has trivial automorphism group also in
characteristic three.

\section{Prime-field families}
\label{sec:prime-field}

We finish with explicit equations over $\F_p$.

\subsection{Odd characteristic}

\begin{proof}[Proof of Theorem~\ref{thm:intro-prime-field}(ii)]
Let $q=p^{2n+1}$ and choose $t\in\F_{p^2}^\ast$ with $t^p=-t$.
Then $t^q=-t$.  Taking $r=p$, $m=2n+1$, and $\alpha=1$ in
\eqref{eq:R-alpha-t} gives
\begin{equation}\label{eq:R-even-support}
 R(X)=tX+2t\sum_{i=1}^{n}X^{p^{2i}}.
\end{equation}
Here $t^p+t=0$, so the assumption $t^r+t\ne0$ of
Corollary~\ref{cor:maximal-parent} does not hold.  Nevertheless,
Lemma~\ref{lem:trace-family} gives the product of the two independent
linear forms $T(x)$ and $T(tx)$, and the leading coefficient of $R$ is
$2t\ne0$.  Proposition~\ref{prop:maximal-rank-two} therefore shows directly
that the covering curve is $\F_{q^2}$-maximal.  All nonzero indices of $R$
are even, so we take
\[
 D=2(p-1).
\]
Proposition~\ref{prop:quotient} gives
\[
 uz^p-z=t\left(1+2\sum_{i=1}^{n}
 u^{(p^{2i}-1)/(2(p-1))}\right).
\]
Putting $z=tv$ and using $t^p=-t$ yields \eqref{eq:intro-odd}.

Its genus is $(p^{2n}-1)/4$, and for $n\ge2$ one has
\[
 \frac{p^{2n}-1}{4}>(p-1)^2,
 \qquad
 2(p-1)>2=\gcd(2,2(p-1)).
\]
Theorem~\ref{thm:rigidity} gives the trivial automorphism group.  The scalar
group and the quotient map are defined over $\F_{q^2}$, so the quotient is
maximal.
\end{proof}

In particular, $p=3$ is included.  For $n=2$ one obtains
\[
 uv^3+v=u^2+u^{20}-1,
\]
a maximal curve over $\F_{3^{10}}$ of genus $20$ with trivial geometric
automorphism group.

\subsection{Even characteristic}

\begin{proof}[Proof of Theorem~\ref{thm:intro-prime-field}(i)]
Take
\[
 r=4,
 \qquad
 m=3n,
 \qquad
 q=4^{3n}=2^{6n}.
\]
Choose $t\in\F_{64}^\ast$ of order $9$ and put $\omega=t^3$.  Then
$\omega\in\F_4\setminus\F_2$ and $\omega^2+\omega=1$.  Since
$\F_{64}\subseteq\F_q$, one has $t^q=t=-t$, and $t\notin\F_4$.
Theorem~\ref{thm:intro-maximal}, with $\alpha=1$, gives a quotient curve
with affine coordinates $u_0,z_0$ and equation
\begin{equation}\label{eq:binary-before}
 u_0^2z_0^4+z_0
 =t+\sum_{j=1}^{3n-1}(t^{4^j}+t)u_0^{N_j},
 \qquad
 N_j=\frac{4^j-1}{3}.
\end{equation}
This curve is maximal over $\F_{2^{12n}}$, has genus $2^{6n-3}$, and has
trivial geometric automorphism group.

Now set
\[
 u_0=\omega u,
 \qquad
 z_0=tv.
\]
Since $\omega^2t^4=t$, equation \eqref{eq:binary-before} becomes
\[
 u^2v^4+v
 =1+\sum_{j=1}^{3n-1}
 (t^{4^j-1}+1)\omega^{N_j}u^{N_j}.
\]
Using $t^{4^j-1}=\omega^{N_j}$ and $N_j\equiv j\pmod3$, we obtain
\[
 (t^{4^j-1}+1)\omega^{N_j}
 =\omega^{2j}+\omega^j
 =\begin{cases}
 0,&3\mid j,\\
 1,&3\nmid j.
 \end{cases}
\]
This is precisely \eqref{eq:intro-binary}.
\end{proof}

For $n=2$ the equation is
\[
 u^2v^4+v=1+u+u^5+u^{85}+u^{341}.
\]
It is defined over $\F_2$, is maximal over $\F_{2^{24}}$, has genus $512$,
and has trivial geometric automorphism group.

\begin{remark}
The two parts of Theorem~\ref{thm:intro-prime-field} show that, for every
prime $p$, there is an explicit sequence of curves defined over $\F_p$, of
unbounded genus, whose geometric automorphism group is trivial and which are
maximal over explicitly prescribed finite extensions of $\F_p$.
\end{remark}

\section{Hermitian coverings}
\label{sec:hermitian-coverings}

Let $\mathcal H_q$ be the Hermitian curve over $\F_{q^2}$, given by
$Y^q+Y=X^{q+1}$.  For each maximal curve $\W_R$ constructed above,
\cite[Theorem~6.12]{CakcakOzbudak} gives an $\F_{q^2}$-Galois
covering $\mathcal H_q\to\W_R$.  The covering
$\W_R\to\X_R$ is also defined over $\F_{q^2}$.
Consequently every explicit maximal curve $\X_R$ in this paper is
$\F_{q^2}$-covered by $\mathcal H_q$.

We recall the results of \cite{CakcakOzbudak} needed below,
including the formulas used to construct the covering.

\begin{proposition}\label{prop:CO-covering}
Let $q=r^m$, where $r>2$ is a prime power and $m\ge2$.
Let $R\in\F_{q^2}[X]$ be $r$-linearized of degree $r^h$,
with $1\le h\le m$, and assume that
\[
 \W_R:\qquad y^r-y=xR(x)
\]
is $\F_{q^2}$-maximal.
Then there exists an $\F_{q^2}$-Galois covering
\[
 \pi:\mathcal H_q\longrightarrow\W_R,
 \qquad \deg\pi=r^{2m-h-1}.
\]
It can be chosen so that
\[
 \pi\circ s_a=\rho_a\circ\pi
 \qquad\text{for every }a\in\F_r^\ast,
\]
where
\[
 s_a(X,Y)=(aX,a^2Y),
 \qquad
 \rho_a(x,y)=(ax,a^2y).
\]
\end{proposition}

\begin{proof}
The existence and degree of the Galois covering follow from
\cite[Theorem~6.12]{CakcakOzbudak}.
We recall its construction to verify the last assertion.

By \cite[Proposition~5.1 and Theorem~5.17]{CakcakOzbudak},
$\W_R$ is covered by a curve
\[
 z^r-z=b t^{q+1},
 \qquad b\in\F_{q^2}^\ast,\qquad b^q=-b.
\]
The covering is a composition of $m-h$ maps of degree $r$,
each having the form
\[
 (t,s)\longmapsto
 \bigl(t^r+ct,\ s+tL(t)\bigr),
\]
where $c\in\F_{q^2}^\ast$ and
$L\in\F_{q^2}[X]$ is $r$-linearized.
By \cite[Proposition~6.5]{CakcakOzbudak}, the initial map
from $\mathcal H_q$ is
\[
 (X,Y)\longmapsto
 \left(X,-\sum_{j=0}^{m-1}(bY)^{r^j}\right).
\]

For $a\in\F_r^\ast$, we have
\[
 (at)^r+c(at)=a(t^r+ct),
 \qquad
 a^2s+(at)L(at)=a^2\bigl(s+tL(t)\bigr).
\]
Moreover,
\[
 -\sum_{j=0}^{m-1}(ba^2Y)^{r^j}
 =-a^2\sum_{j=0}^{m-1}(bY)^{r^j}.
\]
Thus every map in the composition respects the substitutions
$(t,s)\mapsto(at,a^2s)$, proving
$\pi\circ s_a=\rho_a\circ\pi$.
\end{proof}

\begin{proposition}\label{prop:galois-D-rminus1}
If $\W_R$ is $\F_{q^2}$-maximal and $D=r-1$, then
$\X_R$ is an $\F_{q^2}$-Galois subcover of $\mathcal H_q$.
\end{proposition}

\begin{proof}
Let $\pi:\mathcal H_q\to\W_R$ be the Galois covering given by
Proposition~\ref{prop:CO-covering}, and let $C$ be its Galois group.
For every $a\in\F_r^\ast$, the relation
$\pi\circ s_a=\rho_a\circ\pi$ implies that $s_a$ normalizes $C$.

Since $D=r-1$, the automorphisms $\rho_a$, with
$a\in\F_r^\ast$, form precisely the scalar group $C_D$.
Hence the subgroup
\[
 G=\langle C,s_a:a\in\F_r^\ast\rangle
 \le\Aut_{\F_{q^2}}(\mathcal H_q)
\]
has fixed field
\[
 \F_{q^2}(\mathcal H_q)^G
 =\F_{q^2}(\W_R)^{C_D}
 =\F_{q^2}(\X_R).
\]
Therefore $\X_R$ is an $\F_{q^2}$-Galois subcover of
$\mathcal H_q$.
\end{proof}

When $D=2(r-1)$, the conclusion of the preceding proposition
need not hold. We show this for the prime-field curves $\X_{p,n}$
in every odd characteristic and for every $n\ge2$.

Let $p$ be an odd prime, let $m=2n+1\ge5$, and put
$q=p^m$. For $1\ne\sigma\in\mathrm{PGU}(3,q)$, let
\[
 i(\sigma)=\sum_{P\in\mathcal H_q}
 v_P\bigl(\sigma(t_P)-t_P\bigr),
\]
where $t_P$ is a local parameter at $P$. Thus, for
$G\le\mathrm{PGU}(3,q)$, the degree of the different of
$\mathcal H_q\to\mathcal H_q/G$ is
$\sum_{1\ne\sigma\in G}i(\sigma)$.

\begin{lemma}\label{lem:uniform-subspaces}
Let $V\subseteq\F_{q^2}$ be an $\F_p$-subspace which is stable under
multiplication by some $\lambda\in\F_{q^2}\setminus\F_q$, and put
\[
 V_0=\left\langle u^qv-uv^q:u,v\in V\right\rangle_{\F_p}.
\]
Then
\[
 \dim_{\F_p}V\le2\dim_{\F_p}V_0.
\]
\end{lemma}

\begin{proof}
If \(V=\{0\}\), the assertion is immediate. Assume \(V\ne\{0\}\), and fix \(u\in V\setminus\{0\}\). For $v\in V$, both
\[
 u^qv-uv^q
 \quad\text{and}\quad
 \lambda u^qv-\lambda^quv^q
\]
belong to $V_0$. Their difference after multiplication of the first
one by $\lambda^q$ is $(\lambda-\lambda^q)u^qv$. Hence
\[
 u^qV\subseteq(\lambda-\lambda^q)^{-1}(V_0+\lambda^qV_0).
\]
Multiplication by $u^q$ is injective, and the assertion follows.
\end{proof}

\begin{lemma}\label{lem:uniform-three-points}
If $\sigma\in\mathrm{PGU}(3,q)$ permutes cyclically three
noncollinear points of $\PP^2(\K)$, then $\sigma$ has order $3$ and
is neither a homology nor an elation.
\end{lemma}

\begin{proof}
In the basis determined by the three points, a lift of $\sigma$ is a
monomial matrix with nonzero entries $c_1,c_2,c_3$. Its cube is the
scalar matrix $c_1c_2c_3I$, so $\sigma$ has order $3$. The minimal
polynomial of the lift is $T^3-c_1c_2c_3$ and has degree $3$, whereas,
up to multiplication by a scalar, the minimal polynomial of a homology
or an elation has degree $2$.
\end{proof}

\begin{proof}[Proof of Theorem~\ref{thm:uniform-nongalois}]
Suppose that $\X_{p,n}\simeq\mathcal H_q/G$ for some
$G\le\Gamma:=\mathrm{PGU}(3,q)$. Recall that
\[
 g(\X_{p,n})=\frac{p^{m-1}-1}{4}=\frac{q-p}{4p},
 \qquad
 \Aut_{\K}(\X_{p,n})=1.
\]

\medskip
\noindent\emph{Numerical and normalizer reductions.}
The action of $N_\Gamma(G)$ on $\K(\mathcal H_q)^G$ has kernel $G$.
Therefore
\[
 N_\Gamma(G)/G\hookrightarrow\Aut_{\K}(\X_{p,n}),
\]
and hence
\begin{equation}\label{eq:uniform-self}
 N_\Gamma(G)=G.
\end{equation}
If $C\le G$ and $\mathcal H_q/C$ is rational, then
$\mathcal H_q/G$ is rational by L\"uroth's theorem, which is impossible.

Riemann--Hurwitz and Hilbert's different formula give
\begin{equation}\label{eq:uniform-RH}
 (q+1)(q-2)
 =\frac{q-5p}{2p}|G|+
 \sum_{1\ne\sigma\in G}i(\sigma).
\end{equation}
For every nonidentity automorphism of $\mathcal H_q$, one has
$0\le i(\sigma)\le q+2$. It follows that
\[
 \frac{2pq^2}{(2p+1)q-p}
 \le |G|\le
 \frac{2p(q+1)(q-2)}{q-5p}.
\]
Since $q\ge p^5$, we may use the simpler bounds
\begin{equation}\label{eq:uniform-size}
 \frac{6q}{7}<|G|<3pq.
\end{equation}

Suppose that $G$ fixes a point
$P\in\PP^2(\F_{q^2})\setminus\mathcal H_q$. The central group
$C_{q+1}$ of homologies with center $P$ lies in $N_\Gamma(G)=G$.
Since $\mathcal H_q/C_{q+1}$ is rational, this is impossible. Thus
\begin{equation}\label{eq:uniform-noexternal}
 G\text{ fixes no point of }
 \PP^2(\F_{q^2})\setminus\mathcal H_q.
\end{equation}

\medskip
\noindent\emph{Subgroups fixing a point of $\mathcal H_q$.}
Assume that $G$ fixes an $\F_{q^2}$-rational point of
$\mathcal H_q$. After conjugation, this point is the point at infinity
of the model $Y^q+Y=X^{q+1}$. The elements of its stabilizer have the
form
\[
 [\lambda,b,c]:(X,Y)\longmapsto
 (\lambda X+b,\lambda^{q+1}Y+\lambda b^qX+c),
 \qquad c^q+c=b^{q+1}.
\]
Put
\[
 Q=\{[1,b,c]\},\qquad
 Z=\{[1,0,c]\},
\]
so that $Z$ is the center of $Q$ and $|Z|=q$. Set
\[
 \begin{aligned}
 G_1&=\{\lambda:[\lambda,b,c]\in G\},\\
 G_2&=\{b:[1,b,c]\in G\text{ for some }c\},\\
 G_3&=G\cap Z,
 \end{aligned}
\]
and write
\[
 |G_1|=A,\qquad |G_2|=p^j,\qquad |G_3|=p^k,
 \qquad d=\gcd(A,q+1).
\]
Then $|G|=Ap^{j+k}$ and $A\mid q^2-1$.

We first compute the degree of the different. By
\cite[Equation 2.12]{GSX}, if
$\sigma=[\lambda,b,c]\ne1$, then
\[
 i(\sigma)=
 \begin{cases}
 q+2,&\lambda=1, b=0,\\
 2,&\lambda=1, b\ne0.
 \end{cases}
\]
If $\lambda\ne1$, then $\sigma$ is conjugate by an element of $Q$ to
$[\lambda,0,c']$, where $c'^q+c'=0$. When
$\lambda^{q+1}\ne1$, it is conjugate to
$\operatorname{diag}(\lambda,\lambda^{q+1},1)$ and contributes $2$.
When $\lambda^{q+1}=1$, it is either a homology, contributing $q+1$,
or an element whose order is divisible by $p$ and greater than $p$,
contributing $1$.

The elements of $G$ satisfying $\lambda^{q+1}=1$ form a subgroup
with normal Sylow $p$-subgroup $G\cap Q$ and cyclic quotient of order
$d$. For each nonidentity value of $\lambda$ in this quotient, the
corresponding coset contains $p^j$ homologies and
$p^{j+k}-p^j$ elements contributing $1$. Indeed, by
Schur--Zassenhaus the complements are conjugate under $G\cap Q$, and
the centralizer in $G\cap Q$ of $[\lambda,0,0]$ is $G\cap Z$.
Consequently
\[
 \sum_{1\ne\sigma\in G}i(\sigma)
 =q\bigl(p^k-1+(d-1)p^j\bigr)-2
  +p^{j+k}(2A-d+1).
\]
Substitution in \eqref{eq:uniform-RH} gives
\begin{equation}\label{eq:uniform-point}
 2p(q-p^k)\bigl(q-(d-1)p^j\bigr)
 =Ap^{j+k}(q-p).
\end{equation}

Let $\lambda$ generate $G_1$. Conjugation by an element with scalar
part $\lambda$ shows that $G_2$ is stable under multiplication by
$\lambda$, while $G_3$ is stable under multiplication by
$\lambda^{q+1}$. Moreover,
\[
 [1,b,c][1,b',c'][1,b,c]^{-1}[1,b',c']^{-1}
 =[1,0,b^qb'-bb'^q].
\]
Hence
\begin{equation}\label{eq:uniform-commutator}
 \left\langle u^qv-uv^q:u,v\in G_2\right\rangle_{\F_p}
 \subseteq G_3.
\end{equation}

We have $k<m$, since $k=m$ would imply $Z\le G$, whereas
$\mathcal H_q/Z$ is rational. We also have $j>0$. Indeed, if $j=0$,
the action of $G$ on the tangent line at the fixed point factors
through the cyclic scalar quotient. If this quotient is trivial, the
line is fixed pointwise. Otherwise, in the coordinate $Y/X$, a
generator has the form
\[
 Y/X\longmapsto\lambda^qY/X+b^q
\]
and fixes an $\F_{q^2}$-rational point of the tangent line different
from the point of tangency. In both cases this contradicts
\eqref{eq:uniform-noexternal}.

We next show that $k>0$. If $k=0$, reduction of
\eqref{eq:uniform-point} modulo $q-1$ gives
\[
 q-1\mid A(p-1).
\]
Put $B=(q-1)/(p-1)$. Then $B\mid A$ and
$\operatorname{ord}_B(p)=m$. Since $G_2$ is stable under a generator
of $G_1$, it is a vector space over
$\F_p(\lambda)=\F_{p^{\operatorname{ord}_A(p)}}$. Therefore
\[
 m=\operatorname{ord}_B(p)\mid\operatorname{ord}_A(p)\mid j.
\]
It follows that
\[
 |G|=Ap^j\ge Bp^m
 =q\frac{q-1}{p-1}>3pq,
\]
contrary to \eqref{eq:uniform-size}.

If $d=1$, comparison of the $p$-adic valuations in
\eqref{eq:uniform-point} gives $j=m$, and cancellation yields
\[
 2(p^{m-k}-1)=A(p^{m-1}-1).
\]
The only possible equality is $k=1$ and $A=2$, but then
$d=\gcd(2,q+1)=2$. Hence $d>1$.

The second factor on the left-hand side of
\eqref{eq:uniform-point} is positive, and therefore $j<m$.
The action of the scalar group on $G_2$ and $G_3$ gives
\begin{equation}\label{eq:uniform-orders}
 A\mid p^j-1,\qquad A/d\mid p^k-1.
\end{equation}
The equality $A=d$ is impossible: in this case the scalar elements
have norm $1$, so the full group $Z$ normalizes $G$; by
\eqref{eq:uniform-self}, this would give $Z\le G$. Thus $A>d>1$.
Dividing \eqref{eq:uniform-point} by $p^{j+k+1}$ gives
\begin{equation}\label{eq:uniform-point-reduced}
 2(p^{m-k}-1)(p^{m-j}-d+1)=A(p^{m-1}-1).
\end{equation}

If $d=2$, put $a=m-k$ and $b=m-j$. Then
\begin{equation}\label{eq:uniform-d2}
 2(p^a-1)(p^b-1)=A(p^{m-1}-1),
 \qquad 1\le a,b\le m-1.
\end{equation}
By the Bang--Zsigmondy theorem \cite{Zsigmondy}, the integer
$p^{m-1}-1$ has a prime divisor whose multiplicative order for $p$
is $m-1$: here $m-1\ge4$ and $p$ is odd, so there is no
exceptional case. This prime divides $p^a-1$ or $p^b-1$, and hence
$a=m-1$ or $b=m-1$. If $a=m-1$, then $k=1$ and
$A=2(p^b-1)$. By \eqref{eq:uniform-orders}, $p^b-1\mid p-1$,
so $b=1$, $j=m-1$, and $A=2(p-1)$. Since $m$ is odd,
$(q-1)/(p-1)$ is odd, and hence $A\nmid q-1$. Thus the scalar
generator $\lambda$ lies outside $\F_q$. By
Lemma~\ref{lem:uniform-subspaces} and \eqref{eq:uniform-commutator},
$j\le2k=2$, contrary to $j=m-1\ge4$. If $b=m-1$, then $j=1$
and $A=2(p^a-1)\ge2(p-1)$, contrary to $A\mid p^j-1$.

Suppose now that $d>2$, and put $g=\gcd(m,j)$. Since
$d\mid p^m+1$ and $d\mid p^j-1$, the integer $j/g$ is even,
and
\begin{equation}\label{eq:uniform-gcd}
 d\mid p^g+1,\qquad A/d\mid p^g-1.
\end{equation}
First assume $g=1$. Then $A\mid p^2-1$, $d\mid p+1$, and
$A/d\mid p-1$. If $j,k\le m-2$ and $p\ge5$, reduction of
\eqref{eq:uniform-point-reduced} modulo $p^2$ gives
\[
 A+2d\equiv2\pmod{p^2}.
\]
Since $A+2d<3p^2$, write $A+2d=2+\ell p^2$ with
$\ell\in\{1,2\}$. As $d\mid A$ and $d\mid p^2-1$, we get
$d\mid\ell+2$ and hence $d\le4$. Therefore
\[
 A+2d\le4(p-1)+8=4p+4<p^2+2,
\]
a contradiction. For $p=3$, the conditions $A>d$ and
$A/d\mid p-1$ imply $A/d=2$. Since $d\mid4$ and
$d=\gcd(A,q+1)=\gcd(A,4)$, one has $d=4$ and $A=8$.
Equation \eqref{eq:uniform-point-reduced} becomes
\[
 2(3^{m-k}-1)(3^{m-j}-3)=8(3^{m-1}-1),
\]
whose sides have $3$-adic valuations $1$ and $0$, respectively.
Thus $j=m-1$ or $k=m-1$.

If $k=m-1$ and $j\ge2$, equation
\eqref{eq:uniform-point-reduced} gives $A<2$. If $k=m-1$
and $j=1$, the $\F_p$-space $G_2$ is stable under multiplication by
$\lambda$, so $A\mid p-1$, whereas
\[
 A=2(p-1)\frac{p^{m-1}-d+1}{p^{m-1}-1}>p-1.
\]
Thus $j=m-1$. If $k\ge2$, equation
\eqref{eq:uniform-point-reduced} gives $A<3$. If $k=1$, then
$\lambda\notin\F_q$, since otherwise $d\le2$; the same subspace
lemma gives $m-1=j\le2k=2$. This excludes $g=1$.

Finally, assume $g>1$. Since $m$ is odd, $g\ge3$ and $j\ge2g$.
Moreover $\lambda\notin\F_q$, because $d>2$. If $k=1$, then
Lemma~\ref{lem:uniform-subspaces} and
\eqref{eq:uniform-commutator} give $j\le2$, a contradiction.
Suppose $k\ge2$, and put
\[
 h=\gcd(m-1,k-1),\qquad m-1=vh,\qquad
 M=\frac{p^{m-1}-1}{p^h-1},\qquad b=m-j.
\]
Here $v\ge2$. Since
$\gcd(p^{m-1}-1,p^{m-k}-1)=p^h-1$,
equation \eqref{eq:uniform-point-reduced} implies
\begin{equation}\label{eq:uniform-divisibility}
 M\mid2(p^b+1-d),\qquad p^b+1-d>0.
\end{equation}
If $j>h+1$, then $b<(v-1)h$, and hence
$0<2(p^b+1-d)<2p^b<M$, a contradiction. Thus
$c=h+1-j\ge0$ and $b=(v-1)h+c$. Write
$M=p^{(v-1)h}+S$, where
$0<S<\tfrac32p^{(v-2)h}$. Then
\[
 2(p^b+1-d)=2p^cM-2(p^cS+d-1).
\]
By \eqref{eq:uniform-gcd}, $d-1\le p^g$; also
$h\ge j-1\ge2g-1$. Consequently
\[
 0<\frac{2(p^cS+d-1)}{M}
 <3p^{1-j}+2p^{g-h}
 \le\frac3{3^5}+\frac2{3^2}<1,
\]
contrary to \eqref{eq:uniform-divisibility}. This excludes all
subgroups fixing a point of $\mathcal H_q$.

\medskip
\noindent\emph{Self-polar triangles.}
Suppose that $G$ stabilizes a self-polar triangle. Its image in the
permutation group of the three vertices cannot have order $1$ or $2$,
since in that case $G$ would fix a vertex, contrary to
\eqref{eq:uniform-noexternal}. Thus its image has order $3$ or $6$.

Assume first that the image has order $3$, and put
\[
 e=|G\cap(C_{q+1}\times C_{q+1})|.
\]
For $p=3$, a $3$-cycle acts on $C_{q+1}^2$ without nonidentity
fixed points, since the determinant of $1-s$ on this group is $3$
and $3\nmid q+1$. Hence the elements outside the diagonal subgroup
are regular unipotents and contribute $2$ to the different.
By \cite[Proposition~3.5]{DVMZ}, the genus equation becomes
\begin{equation}\label{eq:uniform-p3-C3}
 e(q-7)=2(q+1)(q+1-3a),
\end{equation}
where $a$ is the order of the homology group at each vertex and
$a^2\mid e$. Write $q+1=4m_0$, with $m_0$ odd. Reducing
\eqref{eq:uniform-p3-C3} modulo $q+1$ shows that $m_0\mid e$;
write $e=bm_0$. The upper bound on $|G|$ preceding
\eqref{eq:uniform-size} gives $b\le8$. Equation
\eqref{eq:uniform-p3-C3} reduces to
\[
 (8-b)m_0=6a-2b,\qquad a\le\sqrt{bm_0}\le\sqrt{8m_0}.
\]
For $m\ge7$ one has $m_0\ge547>288$, making the equation
impossible for $b\le7$; if $b=8$, it gives $a=8/3$.
For $m=5$, one has $m_0=61$ and $59b+6a=488$. Modulo $6$
this gives $b=4$, and then $a=42$, contrary to $a^2\mid e=244$.

For $p\ge5$, \cite[Propositions~3.5 and~3.6]{DVMZ} gives a
divisor $a$ of $q+1$ such that $a^2\mid e$, $e/a\mid q+1$, and
\[
 g(\mathcal H_q/G)
 =\frac{(q+1)(q-3a+1)+ce}{6e},
 \qquad c\in\{0,2,6\}.
\]
Equating this with $(q-p)/(4p)$ gives
\[
 e\bigl(3q-(3+2c)p\bigr)
 =2p(q+1)(q+1-3a).
\]
Write $e=a^2b$ and $q+1=abt$. Then
\begin{equation}\label{eq:uniform-C3-necessary}
 bt(3a-2pt)=(3+2c)p+3-6pt.
\end{equation}
In particular, $t\mid(3+2c)p+3$, and hence $t<18p$.
Put $\rho=((3+2c)p+3)/t$. Since $p>3$, one has
$\rho\ne6p$; equation \eqref{eq:uniform-C3-necessary} reads
$b(3a-2pt)=\rho-6p$. It follows that
$b<18p$ and $a<18p^2$. Thus
\[
 q+1=abt<5832p^4.
\]
\begin{samepage}
As $m$ is odd and $m\ge5$, this leaves only $m=5$ with
$5\le p<5832$, $m=7$ with $5\le p<19$, and $m=9$ with
$p=5$. Put $J=3+2c$. Eliminating $a$ from
\eqref{eq:uniform-C3-necessary} and $q+1=abt$ gives the necessary
conditions
\[
 t\mid Jp+3,\qquad
 2bt^2=3p^{m-1}-J+6t,\qquad
 bt\mid p^m+1.
\]
An elementary check of these finite divisibility conditions in
MAGMA gives no solutions for the remaining values of $p$, $m$ and
$c\in\{0,2,6\}$.
\end{samepage}

Assume now that the permutation image is $S_3$. If $p=3$,
\cite[Proposition~3.7]{DVMZ} gives $e=a^2$ for some
$a\mid q+1$. With $z=(q+1)/a$, the genus equation is
\[
 \begin{cases}
  q+1=(z-3)^2+3,&a\mid(q+1)/2,\\
  q+1=(z-3)^2,&a\nmid(q+1)/2.
 \end{cases}
\]
In the first case $z\mid12$ and $q+1\le84$; in the second,
$z\mid9$ and $q+1\le36$. Both contradict $q+1\ge244$.

Let $p\ge5$. By \cite[Proposition~3.7]{DVMZ}, one has
$e=\kappa a^2$ with $\kappa\in\{1,3\}$. Putting
$z=(q+1)/a$ in the quotient-genus formula gives
\begin{equation}\label{eq:uniform-S3}
 3\kappa p^{m-1}
 =z^2-3(1+\kappa)z+\kappa(15-\delta),
 \qquad z\mid q+1,
\end{equation}
where $\delta\in\{0,3,4,7,8,11\}$. Put
$P=p^{(m-1)/2}$. If $\kappa=3$, then
\[
 (z-6)^2=9P^2+3(\delta-3).
\]
Since $P\ge25$ and $-9\le3(\delta-3)\le24$, the only
possible square is $9P^2$. Hence $\delta=3$ and
$z=3(P+2)$. But $z\mid pP^2+1$ implies
$P+2\mid4p+1$, whereas $P+2\ge p^2+2>4p+1$.

If $\kappa=1$, multiplication of \eqref{eq:uniform-S3} by $p$
and reduction modulo $z$ give $z\mid p(15-\delta)+3$.
On the other hand,
\[
 (z-3)^2=3P^2+\delta-6.
\]
Thus $z\le15p+3$ is impossible for $p\ge11$ or $m\ge7$.
Only $(m,p)=(5,5)$ or $(5,7)$ remains. If $p=5$, then
$47\le z\le78$, but $5^5+1=6\cdot521$ has no divisor in
this interval. If $p=7$, then $88\le z\le108$; the only
divisor of $7^5+1=2^3\cdot11\cdot191$ in this interval
is $88$. This would give
$(z-3)^2-3p^4=22\notin\{-6,-3,-2,1,2,5\}$.

\medskip
\noindent\emph{Singer triangles.}
Put $L=q^2-q+1$. A Singer-triangle stabilizer has the form
$C_L\rtimes C_3$. If $G\le C_L$, then $G$ is characteristic in
$C_L$ and is normalized by the complement $C_3$, contrary to
\eqref{eq:uniform-self}. Hence
\[
 |G|=3e,
 \qquad e=|G\cap C_L|\mid L.
\]
No nonidentity element of $G$ is a homology or an elation: the
nonidentity elements of $C_L$ have three distinct eigenvalues, while
the remaining elements cyclically permute three noncollinear points
and Lemma~\ref{lem:uniform-three-points} applies.
If $p=3$, a generator of the complement acts on $C_L$ by
$x\mapsto x^{q-1}$. It has no nonidentity fixed point, since
$\gcd(q-2,L)=1$. The $e-1$ elements in $C_L\setminus\{1\}$
contribute $3$ each, and the other $2e$ elements are regular
unipotents and contribute $2$ each. Substituting their contributions
in \eqref{eq:uniform-RH} gives
\[
 e\frac{q-1}{2}=q(q-1)+1.
\]
Thus $q-1\mid2$, which is impossible.

Let $p\ge5$. By \cite[Proposition~4.2]{DM}, every nonidentity
element contributes at most $3$, so
$0\le\sum_{1\ne\sigma\in G}i(\sigma)\le9e-3$.
Riemann--Hurwitz gives
\[
 6eg(\X_{p,n})=L-3+6e-
 \sum_{1\ne\sigma\in G}i(\sigma).
\]
With $z=L/e$, it follows that
\[
 c:=3p^{m-1}-2z
 =3+\frac{2}{e}
 \left(6e-3-\sum_{1\ne\sigma\in G}i(\sigma)\right)
 \in\{-3,-2,\ldots,14\}.
\]
Using $L=p^{2m}-p^m+1$ and reducing modulo $z$ gives
\begin{equation}\label{eq:uniform-Singer-div}
 z\mid c^2p^2-3cp+9<196p^2.
\end{equation}
Since $2z=3p^{m-1}-c\ge3p^{m-1}-14$, this is
impossible for $m\ge7$ or $p\ge13$. Thus $m=5$ and
$p\in\{5,7,11\}$. As $p$ is odd, $c$ must be odd, so
$c\in\{-3,-1,1,3,5,7,9,11,13\}$. If $p=11$,
then $0<c^2p^2-3cp+9<z$ for each value of $c$.
If $p=7$, the lower inequality
$c^2p^2-3cp+9\ge z$ leaves only $c=9,11,13$.
Writing $c^2p^2-3cp+9=az$ gives respectively
$a\le1,1,2$. Reduction modulo $p$ would give
$ac\equiv-18\pmod p$, which is impossible in these cases.
If $p=5$, the same congruence leaves only
$(c,a)=(7,1),(11,2),(13,4)$. The corresponding values of
$c^2p^2-3cp+9-az$ are $195,1005,315$, respectively, all
nonzero.

\medskip
\noindent\emph{The remaining maximal subgroups.}
The upper bound in \eqref{eq:uniform-size} shows that $G$ is a proper
subgroup of $\Gamma$. After the analysis above, we can assume now that $G$ stabilizes no point or triangle; the remaining possibilities in the maximal-subgroup classification for $\Gamma$ are described in \cite{MZwithout}; see in particular \cite[Theorem 2.1]{MZwithout} for $G\leq {\rm PSU}(3,q)$ and \cite[Theorem 3.1]{MZwithout} for $G\leq \Gamma\ne{\rm PSU}(3,q)$.

Suppose first that $G$ is contained in a conic stabilizer
$\mathrm{PGL}(2,q)$. By Dickson's classification
\cite[Chapter XII, Paragraph 260]{Dickson} (see also \cite[Theorem A.8]{HKT}), the possibilities
are groups of bounded order, groups with a nontrivial normal
$p$-subgroup, cyclic or dihedral groups contained in the normalizer of
a cyclic group of order $q-1$ or $q+1$, and the groups $\mathrm{PSL}(2,p^f)$ or
$\mathrm{PGL}(2,p^f)$ with $f\mid m$.

The groups of bounded order have order at most $60$, contradicting
\eqref{eq:uniform-size}. A nontrivial normal $p$-subgroup fixes a
unique point of $\mathcal H_q$, and its normality forces $G$ to fix
the same point. The cyclic and dihedral cases are also impossible.
Indeed, the cyclic group of order $q-1$ or $q+1$ has a unique involution, which is central in its
normalizer and therefore normalizes $G$. By
\eqref{eq:uniform-self}, this involution lies in $G$. It is a homology,
and its center is a point outside $\mathcal H_q$ fixed by $G$, contrary
to \eqref{eq:uniform-noexternal}. If $G$ lies between
$\mathrm{PSL}(2,p^f)$ and $\mathrm{PGL}(2,p^f)$, put $t=p^f$.
For $m/f\ge5$, its order is less than $t^3\le q^{3/5}<6q/7$;
for $m/f=1$, its order exceeds $3pq$. If $m/f=3$, then
$q=t^3$, and every proper subgroup of $\mathrm{PGL}(2,t)$ has
order less than $q/2$. Thus $G=\mathrm{PGL}(2,t)$, with
$t\ge27$. The $t^2$ involutions are its only possible
homologies. Its nonidentity unipotent elements act as regular
unipotents in the three-dimensional conic representation, so there
are no elations. Every other nonidentity element contributes at most
$3$ to the different. Its degree is therefore at most
$t^2(q+1)+3(q-t-1)$. As $|G|=q-t$, Riemann--Hurwitz gives
\[
 g(\mathcal H_q/G)>q/3>
 \frac{p^{m-1}-1}{4}=g(\X_{p,n}),
\]
a contradiction.

The exceptional maximal subgroups have order at most $2520$.
For $p\ge5$ or $m\ge9$ this is below $6q/7$. When
$p=3$ and $m=5,7$, the only possible exceptional group is
$\mathrm{PSL}(2,7)$, of order $168$, which is also too small;
the other exceptional groups do not occur in these cases by
\cite[Theorem~2.1]{MZspectrum}.

Finally, suppose that $G$ is contained in a proper unitary subfield
group. Put $t=p^f$ and $m=sf$, with $s\ge3$ odd and $f\mid m$.
The orders of the groups $\mathrm{PSU}(3,t)$ and
$\mathrm{PGU}(3,t)$ lie between
\[
 \frac{t^3(t^3+1)(t^2-1)}{3}
 \quad\text{and}\quad
 t^3(t^3+1)(t^2-1)<t^8.
\]
If $s\ge9$, even the full group is too small. If
$s\in\{3,5,7\}$ and $f\ge3$, both full groups have order
greater than $3pq$. The classification at level $t$ then shows
that a proper subgroup large enough to contain $G$ must lie in
a point stabilizer or, when $s=3$, in a conic stabilizer.
Indeed, the stabilizers of the two kinds of triangles have order
at most $6(t+1)^2$ or $3(t^2-t+1)$; exceptional subgroups
have order at most $2520$; and a further proper unitary subfield
group has order less than $t^{8/3}<6q/7$. A conic stabilizer
has order $t(t^2-1)<t^3$, and for $s\ge5$ it is also too
small. The point stabilizers fix an $\F_{q^2}$-rational point
on or off $\mathcal H_q$, since $s$ is odd. These cases have
already been excluded, as have conic stabilizers.

It remains to take $f=1$. When $m=5$, both full groups are
larger than $3pq$; when $m\ge9$, they are smaller than $6q/7$.
Their proper subgroups are excluded by the same classification:
apart from $\mathrm{PSU}(3,p)$ of index $3$ in
$\mathrm{PGU}(3,p)$ when $3\mid p+1$, any subgroup large enough
to contain $G$ fixes a point or lies in a conic stabilizer.
Thus, for $m=7$, only the full groups remain. If $p\ge5$,
there are fewer than $p^5$ homologies and fewer than $p^4$
elations. Every remaining nonidentity element contributes at
most $3$, so the different has degree at most
\[
 (p^5+p^4)(p^7+2)+3p^8.
\]
As $|\mathrm{PGU}(3,p)|<p^8$, Riemann--Hurwitz gives
$g(\mathcal H_q/G)>p^6/3$, whereas
$g(\X_{p,3})=(p^6-1)/4$. This excludes both full groups.
For $p=3$, one has
$\mathrm{PGU}(3,3)=\mathrm{PSU}(3,3)$, of order $6048$.
By \cite[Proposition~4.9]{MZwithout}, its different has degree
\[
 56(q+2)+189(q+1)+1344+3024+1512+5184=547180.
\]
Hence
\[
 g(\mathcal H_{3^7}/\mathrm{PSU}(3,3))
 =1+\frac{(q+1)(q-2)-547180}{2\cdot6048}=351,
\]
while $g(\X_{3,3})=(3^6-1)/4=182$. This exhausts the
possibilities for $G$ and proves the theorem.
\end{proof}

\section*{Statements and Declarations}

\noindent\textbf{Declaration of generative AI use.}
During the preparation of this manuscript, the authors used ChatGPT
(OpenAI) to improve the exposition, check mathematical
 arguments, and assist in the development of proofs based on ideas provided by the authors.
All results and proofs were independently verified by the authors, who assume full responsibility
for the content of the manuscript.

\section*{Acknowledgements}
The research of Daniele Bartoli, Marco Timpanella, and Giovanni Zini
was partially supported by the INdAM--GNSAGA Project
``Gruppi di automorfismi di curve algebriche in caratteristica positiva''
(CUP E53C25002010001).
The research of Herivelto Borges was supported by CNPq, grant 310194/2022-9, and by FAPESP, grant 2024/00923-6.

\end{document}